\documentclass[a4paper]{amsart}
\usepackage{epsfig}
\usepackage[utf8]{inputenc} 
\usepackage[all]{xy}
\usepackage{amssymb}
\usepackage{amsmath}
\newtheorem{theorem}{Theorem}
\newtheorem{lemma}[theorem]{Lemma}

{\it}{\rm}

\newtheorem{corollary}[theorem]{Corollary}

\newtheorem{proposition}[theorem]{Proposition}
\newtheorem*{bem}{Remark}{\it}{}
\newenvironment{myrem}{\begin{bem}\begin{rm}}{\end{rm}\end{bem}}
\newcommand{\bigoh}{\operatorname*{O}}
\numberwithin{equation}{section}

\newcommand{\Z}{{\mathbb Z}}
\newcommand{\disc}{{\rm disc}}
\newcommand{\lcm}{{\rm lcm}}

\newcommand{\N}{{\mathbb N}}
\newcommand{\R}{{\mathbb R}}
\newcommand{\Q}{{\mathbb Q}}

\newcommand{\A}{{\mathbb A}}

\newcommand{\gen}{{\rm gen}}
\newcommand{\spn}{{\rm spn}}

\newcommand{\tr}{{\rm tr}}
\newcommand{\rpd}{{\rm rpd}}

\newcommand{\fo}{{\mathfrak o}}

\begin{document}

\title[Averages of Fourier coefficients]
{Averages of Fourier coefficients of Siegel modular forms and
  representation of binary quadratic forms by quadratic forms in four variables}  
\author[R. Schulze-Pillot]{Rainer Schulze-Pillot} 
\dedicatory{To the memory of Hiroshi Saito}
 \begin{abstract}
Let $-d$ be a a negative discriminant and let $T$ vary over a set of
representatives of the integral equivalence classes of integral binary
quadratic forms of discriminant $-d$.
We prove an asymptotic formula for $d \to \infty$ for the average over
$T$ of the number of representations 
of $T$  by an integral positive definite
quaternary quadratic form and obtain bounds for averages of 
Fourier coefficients of linear combinations of Siegel theta series. We
also find an asymptotic estimate from below on the 
number of binary forms of fixed discriminant $-d$ which are
represented by a given quaternary form. In particular, we can show
that for growing $d$ a positive proportion of the binary quadratic forms of
discriminant $-d$ is represented by the given quaternary quadratic form.
\end{abstract}

\maketitle

 \section{Introduction.} 
If $A$ is an integral positive definite symmetric $(m\times m)$ matrix
one calls a solution $X\in M(m\times n, \Z)$ of the system of
quadratic  equations of the form
 \begin{equation*}
Q_A(X) =
{^tX}AX = T 
 \end{equation*}
where  $T$ is another (half-) integral symmetric matrix of size $n
\leq m$ an (integral) representation of $T$ by $S$. It is well known
that the local global principle of Minkowski and Hasse is not valid
for integral representations, but if $m$ is large enough compared to
$n$ one can prove that a positive definite $T$ which is represented by $S$ over all $\Z_p$
and is large enough in a suitable sense is indeed represented by $S$
over the rational integers $\Z$, at least under some mild additional
conditions.
The bound on the size of $m$ necessary for this has recently been
pushed down to $m \ge n+3$, again under suitable additional
conditions, in \cite{ev}, see also \cite{rsp} for an attempt to
optimize those additional conditions. The case $m=n+2$ brings some
limitations due to the existence of the so called spinor exceptions
(see \cite{kneser_mz,hsia_spinor}), taking these into account a result
of the desired type could be reached in \cite{duke_sp} for $n=1, m=3$,
i.\ e.\ for representations of sufficiently large numbers by ternary
forms. 
An analogous result for $n=2$ asserting that all binary
quadratic forms of sufficiently large determinant (and perhaps large
minimum) are represented by a given quaternary quadratic form if they
are represented locally everywhere appears
to be out of reach at the moment. However, 
Einsiedler, Lindenstrauss, Michel, and Venkatesh obtained recently a
result (not yet published) for
a certain subset of $T$ by adapting Linnik's ergodic method, which was
invented to deal with the ternary case, to the problem of
representation of binary forms by quaternary forms. They also obtained
a result on representation of $T$ ``on average''  under the assumption
of the so called Linnik condition. 

In this note we prove similar, but somewhat different  average results
starting out from work of Böcherer and the author  \cite{bs_walds,bs_km} on Siegel theta series of quaternary positive definite
quadratic forms (or quadratic lattices) attached to ideals in Eichler
orders  of square free level in definite quaternion 
algebras. For such a quaternary quadratic lattice $\Lambda$
we proved there among others that the average $r_{av}(\Lambda, d)$
over binary symmetric 
matrices $T$ of fixed discriminant $-d$ of the representation numbers
$r(\Lambda, T)$ can be related to the product of the representation
numbers of $d$ by two attached ternary lattices $L,L'$. 

In the present article we show first 
that the results of \cite{bs_walds,bs_km}
can be extended to a more general situation than 
treated there. Whereas in those articles we had to restrict attention
to quaternary quadratic forms of square free level $N$ and determinant
$N^2$, we can now allow more general levels in the case of square
determinant and can also treat forms of prime determinant and certain
forms of square free determinant $\Delta$. In these latter cases the
role of the pair of ternary $\Z$-lattices in \cite{bs_walds,bs_km} is
taken over by one ternary lattice over the integers of  $\Q(\sqrt{\Delta})$. 
Taken together the
results now cover in particular all maximal quaternary quadratic
lattices of prime level. 

An important role in these generalizations is played by ideas of
Hiroshi Saito which he explained to  Böcherer and me in a letter in
2001; they relate  our work to Hijikata's theory of optimal
embeddings of quadratic orders.

We show then that the formulas for averaged representation numbers 
imply almost immediately asymptotic formulas  and
estimates for the 
$r_{av}(\Lambda,d)$ and hence also
for averages of Fourier coefficients of linear combinations of Siegel theta
series of these forms (Yoshida liftings). If such a linear combination
is a cusp form we obtain 
upper bounds for the averaged Fourier coefficients which appear to be new
and are sharper than what can be 
deduced from the known estimates for the individual Fourier
coefficients. As another consequence of the asymptotic formula for the
average representation number $r_{av}(\Lambda,d)$ we can derive a lower
bound for the number of binary
quadratic forms of fixed discriminant which are represented by a given
integral quaternary quadratic form, in particular, we can show that a
positive proportion of these binary forms is represented.

I heard of the results of Einsiedler, Lindenstrauss, Michel, and
Venkatesh in the talk of Michel at the International Colloquium on
Automorphic Representations and $L$-functions 2012 at the Tata
Institute of Fundamental Research, Mumbai, where a first version of
this note was also  written during a stay following that colloquium. I
thank the institute and in particular Dipendra Prasad for their
hospitality. 

\section{Quaternary lattices and quaternion algebras}
Following the work of Brandt and Eichler (see \cite{eichler_qfog}) on
the connection between quadratic spaces of dimension $4$ and their
orthogonal groups on one side and quaternion algebras on the other
side, Ponomarev has studied in \cite{pono_arith} in detail the
correspondence between quaternary quadratic lattices and orders and
ideals in quaternion algebras. Using these results we will consider
the following two situations.

\smallskip
{\bf Case A: Square discriminant}
\smallskip

Let $D$ be a quaternion algebra over $\Q$ with reduced norm $n$
and reduced trace $\tr$,  view $D$ as a
quadratic space with quadratic form $n$ and associated symmetric
bilinear form $b(x,y)=\tr(x\bar{y})$, where $y\mapsto \bar{y}$ is the
usual involution (quaternionic conjugation) on $D$. 
We assume $D$ to be ramified at $\infty$ and denote by $N_0$ the
product of the finite primes $p$ at which $D$ is ramified (i. e.,
$D\otimes \Q_p$ is a division algebra).
The special orthogonal group $SO(D,n)$ of the quadratic space $(D,n)$
is then isomorphic to $\{(x,y)\in D^\times \times D^\times \mid
n(x)=n(y)\}/Z(D^\times)$, where the center $Z(D^\times) =\Q^\times$ is
embedded diagonally into $D^\times\times D^\times$; an isomorphism is
given by sending the class of $(x,y)$ to the map $\sigma_{x,y}$ given
by $\sigma_{x,y}(d)=xdy^{-1}$. This isomorphism extends naturally to
the $v$-adic completions for a valuation $v=v_p$ of $\Q$ and to the
adelization. 
 
Let $R$ be an order (of full rank) in $D$ and consider the
decomposition 
\begin{equation*}
 D_{\A}^\times = \bigcup_{j=1}^hD_{\Q}^\times y_j R_{\A}^\times 
\end{equation*}
of the adelization $D_{\A}^\times$ into double cosets, where
$R_{\A}=D_\infty \times \prod_{p}R_p$ is the adelization of the order
$R$ and where $R_p=R\otimes \Z_p$. The $y_j$ may be chosen to be of
norm $1$ if the $R_p^\times$ have  $\Z_p^\times$ contained in
their norm groups, which will be the case in the sequel.
The ideals $y_iRy_j^{-1}=:I_{ij}$ with left order $R_i=y_iRy_i^{-1}$
and right order $R_j=y_jRy_j^{-1}$ represent  then all isometry
classes of lattices in the genus of $(R,n)$, with some classes
possibly occurring more than once.  

Let $D$ vary over the quaternion algebras over $\Q$ which are ramified
at $\infty$ and let $R\subseteq D$ vary over the orders in $D$ such
that $R_p$ contains the maximal order of the unramified quadratic
extension of $\Q_p$ if $p$ ramifies in $D$ and is an intersection of
at most two maximal orders in $D_p$ if $p$ splits in $D$.
We obtain then in this way representatives of all genera of quaternary
quadratic lattices $(L,q)$ of square discriminant with $q(L)\Z=\Z$, 
for which the Jordan splitting at the  prime $p$ consists of a
binary (even) unimodular and a binary (even) $p^{r_p}$-modular
component for some $r_p$ for all finite primes $p$. These orders are
Eichler orders of level 
$p^{r_p}$ at the split primes $p$ and orders of level $p^{r_p}$ with
odd $r_p$ in the terminology of \cite{pizer_quaternion2} for the
ramified primes. We write $N_1$ for the product of these $p$-adic
levels for the ramified primes, $N_2$ for the product of the
$p$-adic levels for the $p$ at which $D$ splits and $N=N_1N_2$ for the
global level of the order in question.

We will restrict our attention to these orders and lattices in the
sequel.
We further denote for $s\| N$ (i.e., $s\mid N$ with $\gcd(s,N/s)=1$) 
by $I_{ij}^{\#,s}$ the lattice obtained as $I_{ij}^{\#}\cap
\Z[\frac{1}{s}]I_{ij}$, where $\#$ indicates taking the dual lattice, and
by $I_{ij}^{*,s}$ the same lattice with the quadratic forms scaled
by the factor $s$. We notice that this last lattice is in the same
genus as the $I_{ij}$ and is hence isometric to some $I_{i'j'}$. We
call the  $I_{ij}^{*,s}$  the rescaled partial duals of $I_{ij}$ and
write $\rpd(I_{ij}, N')$ for the set of all $I_{ij}^{*,s}$ with $s\|
N'$, where $N'\|N$ is some fixed exact divisor of $N$. 

For any  lattice $\Lambda$ with positive definite quadratic form $q$ and
symmetric bilinear form $b$ and a  positive semidefinite matrix $T=(t_{kl})_{k,l=1..2}\in M^{\rm sym}(2, \Z)$ we write 
\begin{equation*}r(\Lambda,T):=\vert \{(x,y) \in \Lambda
  \mid \begin{pmatrix} 2q(x)&b(x,y)\\b(x,y)&2q(y)\end{pmatrix}=T\}\vert  
\end{equation*}    
for the number of representations of $T$ by the quadratic lattice
$(\Lambda,q)$ and $r^{*}(\Lambda,T)$ for the number of primitive
representations, i.e., representations $(x,y)$ where $\Z x+\Z y$ is a
direct (but not necessarily orthogonal) summand of $\Lambda$ as a $\Z$-module.
This notation will in particular be applied to the lattices
$I_{ij},I_{ij}^{*,s}$ with the norm form $n$ on them.
We write
\begin{equation*}
L_i:=\{x \in \Z \cdot 1 +2R_i\mid \tr(x)=0\}
\end{equation*} and $r(L_i,t)=\vert\{x\in L_i\mid n(x_i)=t\}\vert$ for a
positive integer $t$. The quadratic lattice $(L_i,n)$ has level $4N$
and discriminant $32N^2$. It is easily checked that for a fixed
decomposition $N=N_1N_2$ all the $L_i$ are
in the same genus of quadratic lattices.  
 
\medskip
{\bf Case B: Non-square discriminant}

\smallskip
Let $D$ be a quaternion algebra over $Q$ ramified at $\infty$ and let
$K=\Q(\sqrt{\Delta})$ be a real quadratic field of discriminant
$\Delta$. The quaternion algebra $D_K=D\otimes K$ over $K$ is then
ramified at the two infinite places of $K$ and at all those finite
places of $K$ lying over a prime $p \in \N$ which splits in $K/\Q$ and
for which $D$ is ramified at $p$. We denote by $x\mapsto \bar{x}$ the
extension of the standard involution on $D$ to $D_K$ and by $\tau$ the
extension of the non-trivial Galois automorphism of $K/\Q$ to $D_K$
(with $\tau(d\otimes a)=d \otimes \tau(a)$).

Then 
\begin{equation*}
V:=\{x\in D_K\mid \tau(x)=\bar{x}\},  
\end{equation*}
equipped with the norm form, is a quaternary quadratic space over $\Q$
of determinant $\Delta$, its special orthogonal group is isomorphic
to 
\begin{equation*}
\{\alpha \in D_K \mid n(\alpha)\in \Q^\times\}/\Q^\times,  
\end{equation*}
where the class of $\alpha$ acts via $x \mapsto \alpha x
\tau(\alpha)^{-1}$.
Varying $D$ and $\Delta$ one obtains all positive definite quaternary
quadratic spaces over $\Q$ of non-square discriminant.
For later use we notice that for a prime $p$ ramified in $K/\Q$ the
Witt invariant of $V$ (see e. g. \cite[Ch. V, \S 3]{lam_qf}) is $1$ if $D$
splits at $p$, $-1$ if $D$ is ramified at $p$. 

An order $R$ in $D_K$ is called symmetric if $\tau(\bar{R})=R$, in
that case $R\cap V=:\Lambda$ is a $\Z$-lattice of rank $4$ on $V$, and the
genus of $\Lambda$ consists of the isometry classes of the lattices $y \Lambda
\tau(y)^{-1}=y R \tau(y)^{-1}\cap V$ for $y \in D_{K,\A}^1:=\{y\in
D_{K,\A}^\times\mid n(y)\in \Q_{\A}^\times\}$, where the extension of $\tau$ to the
adeles of $K$ interchanges the components for the two places of $K$
above a rational prime which splits in $K/\Q$ and where we identify
$\Q_\A^\times$ with its image under the usual embedding into
$K^\times_\A$. 
If $y$ runs here
through the $y_i$ in a double coset decomposition $D_{K,\A}^1=\cup_{i=1}^hD^1_K
y_i R^1_\A$ (with $R^1_\A=D_{K,\A}^1\cap R_\A^\times$), one obtains
all classes in the genus of $\Lambda$ (not necessarily only once). The
$y_i$, having norm in $\Q_\A^\times$, can be chosen to have norm $1$
for the orders considered below. We notice
that $I_i:=y_i R \tau(y_i)^{-1}$ is an ideal with left order $R_i=y_i
R y_i^{-1}$ and right order $\tau(R_i)$.

The genus of maximal lattices on $V$ is obtained by taking for $R_{(p)}$
a symmetric Eichler order of level $(p)$ if $D$ is ramified at the rational
prime $p$ and $p$ is inert in $K/\Q$, and a (symmetric) maximal order
of $(D_K)_v$ at all other finite places $v$ of $K$. The determinant of
these maximal lattices is then $\Delta N^2$, where $N$ is the product
of all finite primes which ramify in $D$ but not in $K/\Q$. We will
restrict later to the case where this determinant is square free, odd
and congruent to $1$ modulo $4$.

\section{Averages of representation numbers}
With notations as above and some positive definite quadratic lattice
$\Lambda$ we write for a discriminant $-d$ 
\begin{equation*}
  r_{av}(\Lambda,d)= \sum_T\frac{r(\Lambda,T)}{\epsilon(T)},
\end{equation*}
where the sum is over a set of $SL_2(\Z)$-equivalence classes of integral symmetric matrices
$T=\bigl(\begin{smallmatrix}2a&b\\b&2c\end{smallmatrix}\bigr)$ of
(signed) discriminant $b^2-4ac=-d$ and where $\epsilon (T)$ denotes
the number of proper automorphisms (or units) of $T$, i.\ e. the
number of $U\in SL_2(\Z)$ with ${}^tU TU=T$. The primitive average
$r_{av}^{*}(\Lambda,d)$ is defined analogously by using the primitive
representation numbers $r^{*}(\Lambda,T)$.

In \cite{bs_walds,bs_km} we related the average representation numbers
$r_{av}(I_{ij},d)$ for the (proper) ideals $I_{ij}$ of Eichler orders of square
free level to sums of products $r(L_i,d)r(L_j,d)$ of representation
numbers of the associated ternary lattices introduced in case A in the previous
section. Since we want to establish a similar relation in a more
general context we recall the setup.

With $D,R\subseteq D$ as in Section 2, case A let $k_1\subseteq D$ be
an imaginary quadratic field and let $k_2 \subseteq D$ be isomorphic
to $k_1$. The set 
\begin{equation*}
\{z \in D \mid \bar{z}k_1z\subseteq k_2\}  
\end{equation*}
is then the orthogonal sum $U_1\perp U_2$ of two two-dimensional
subspaces $U_1,U_2$ of $D$ with $k_1U_\nu=U_\nu=U_\nu k_2$ for $\nu=1,2$, and we
can associate to the optimally embedded quadratic suborders
$\fo_1=k_1\cap R$ and $\fo_2=k_2\cap R$ of $R$ the primitive binary
sublattices $M_\nu=U_\nu\cap R \, (\nu=1,2)$ of $R$, which have a
quadratic left order $\fo_1$ and right order $\fo_2$. 

More generally, with
the $R_i, I_{ij}$ as defined above we have for $1 \le i,j\le h$ the
binary lattices $M_\nu=U_\nu\cap I_{ij}$
with associated quadratic left order $\fo_1=k_1\cap R_i$
and quadratic right order $\fo_2=k_2\cap R_j$ contained in the left
order $R_i$ 
resp. in the right order $R_j$ of the ideal $I_{ij}$.
Conversely, given a primitive binary sublattice $M \subseteq I_{ij}$
we obtain its optimally embedded (into $R_i$ resp. $R_j$) quadratic
left and right orders. It is then easy to see that the determinant of
$M$ and the determinants of its associated quadratic orders, viewed as
binary quadratic lattices, generate the same rational square
class. Moreover, considering the ternary lattices $L_i\subseteq R_i$
defined in Section 2 we showed that one has a bijection between $\{x
\in L_i \text{ primitive } \mid n(x)=d\}/\{\pm 1\}$ and optimally
embedded quadratic suborders of discriminant $-d$  in $R_i$, so that
we obtain in summary  a correspondence between pairs of lines in $L_i,
L_j$ on one side and (pairs of) primitive binary sublattices of
$I_{ij}$. So far this construction does not depend on the level of the
quaternion orders involved.

A local computation then showed in the case of square free level that the discriminant of 
of the completion $(M_\nu)_p$ at a prime $p$ is equal to the
$\text{lcm}$ of the discriminants of the associated quadratic orders,
unless $p$ divides the level $N$ and both discriminants are units at
$p$, in which case one of the $(M_\nu)_p$ has the same discriminant as
the quadratic orders, while for  the other one this discriminant is
multiplied by $p^2$.

A partial generalization of this is:

\begin{proposition}
Let $R$ be an order in $D$ such that the completion $R_p$ is
\begin{itemize}
\item an order of level $p^{r_p}$ with $r_p$ odd (see
  \cite{pizer_quaternion2}) for the primes $p$ at which $D$ is
  ramified
\item an Eichler order of level $p^{r_p}$ for the primes $p$ at which
  $D$ splits,
\end{itemize}
with $r_p=0$ for almost all $p$.
Let $N_1, N_2, N$ be as in Section 2, case A.

Let primitive vectors $x \in L_i, x'\in L_j$ be given with $d=n(x),
d'=n(x')$ in the same rational square class and $p \nmid d\cdot d'$
for all $p\mid N$ with $r_p>1$, let $M_1, M_2$ be the two primitive
binary sublattices of $I_{ij}$ described above.

Then 
\begin{eqnarray*}
\disc(M_1)&=&-\lcm(d,d')s_1^2\\
\disc(M_2)&=& -\lcm(d,d')s_2^2,  
\end{eqnarray*}
where $s_1,s_2$ are relatively prime, $s_1 \mid\mid N, s_2 \mid\mid N$
and $p \mid N$ divides $s_1s_2$ if and only if $p\nmid d\cdot d'$. 
\end{proposition}
\begin{proof}
This has been proven for square free $N$ in \cite{bs_walds,bs_km}. It
remains to do the local computations for a prime $p$ with $p^2 \mid N$
and $p\nmid d \cdot d'$.  As in \cite{bs_walds,bs_km} we may assume
$R_i=R_j=R=I_{ij}$.

If $p \mid N_1$ we have $R_p\subseteq R_{p,{\rm max}}$, the unique
maximal order in the division algebra $D_p$, and $R_p$ as a quadratic
lattice is the orthogonal sum of a binary unimodular anisotropic
sublattice $K_1$, isometric to the maximal order of the unique
unramified quadratic extension of $\Q_p$ with its norm form as
quadratic form, and a $p^{r_p}$-scaled copy $K_2$ of this binary
lattice. In particular, the quadratic suborders $\fo, \fo'$ of $R_p$ associated to
$x,x'$ are isomorphic to this unramified quadratic maximal order.

By Lemma 6 b) of \cite{bs_walds}, we have 
$z_1\in R_{p,{\rm  max}}^\times$ with $z_1^{-1}\fo z_1=\fo'$.    
The orthogonal complement $K_1$ of $\fo_p$ in $R_p$ is then mapped by
conjugation with $z_1$ to a $p^{r_p}$-modular  (hence
  $p^{r_p}\Z_p$-maximal) lattice on the anisotropic subspace
  orthogonal to $\fo'_p$. Since the $p^{r_p}$ -maximal lattice on an
  anisotropic space is unique, we have $z_1^{-1}K_1z_1=K_1'$, where
  $K_1'$ is the orthogonal complement in $R_p$ of $\fo'$. In
  particular, we have $z_1^{-1}R_pz_1=R_p$, which by Theorem 4.2 of
  \cite{hi-pi-she} implies $z_1\in R_p^\times$. This gives
  $(M_1)_p=\fo_pz_1$, so $(M_1)_p$ is unimodular and can be split off
  orthogonally in $R_p$. The other lattice $(M_2)_p$ is then a
  primitive lattice on the space complementary to $(M_1)_p$, hence
  isometric to a $p^{r_p}$-scaled copy of $(M_1)_p$.

\smallskip
For $p \mid N_2$ we can copy the proof of \cite[Proposition
2]{bs_walds}, with the only difference that the index of $R_p$ in the
two maximal orders $R_p^0, R_p^1$ is now not $p$ but $p^{r_p}$, and
similarly the index of the ternary lattices occurring there changes
from $p$ to $p^{r_p}$. One obtains as there that one can find a $z\in
R^\times$ with $z^{-1}\fo z=\fo'$, and as above we see that one of
$(M_1)_p,(M_2)_p$ has discriminant $-n(x)$ and the other one
$-p^{2r_p}\cdot n(x)$.
\end{proof}
As in \cite[Theorem 2.1]{bs_km} we obtain from this: 
\begin{theorem}\label{splitcase_average}
Let  $-d$ be a discriminant which is not divisible by any $p\mid N$
with $r_p>1$ and  write $N_d:=\frac{N}{\gcd(N,d)}$. Then 
\begin{equation}\label{kmformula1}
  \sum_{s\mid N_d}\sum_{m^2 \mid d}r_{av}^{*}(I_{ij}^{*,s},\frac{d}{m^2})=r(L_i,d)r(L_j,d)
\end{equation}
or equivalently
\begin{equation}
  \label{kmformula1b}
   \sum_{s\mid N_d}r_{av}^{*}(I_{ij}^{*,s},d)=\sum_{m\in \N,m^2\mid d}\mu(m)r(L_i,\frac{d}{m^2})r(L_j,\frac{d}{m^2}).
\end{equation}
If $N_d$ is a prime power this simplifies to
\begin{equation}\label{kmformula2}
 \sum_{m^2 \mid d}r_{av}^*(I_{ij},\frac{d}{m^2})=
  \begin{cases}
r(L_i,d)r(L_j,d)    & N_d=1\\
\frac{1}{2}r(L_i,d)r(L_j,d)& N_d\ne 1
\end{cases}
\end{equation}
respectively
\begin{equation}
  \label{kmformula2b}
 r_{av}^*(I_{ij},d)=\begin{cases}
\displaystyle \sum_{m\in \N,m^2\mid
  d}\mu(m)r(L_i,\frac{d}{m^2})r(L_j,\frac{d}{m^2})&N_d =1\\
\displaystyle \frac{1}{2}\sum_{m\in \N,m^2\mid
  d}\mu(m)r(L_i,\frac{d}{m^2})r(L_j,\frac{d}{m^2})&N_d \ne 1
\end{cases}. 
\end{equation}
In particular, if in addition  $-d$ is fundamental the sum over
$m^2\mid d$ can be omitted in the above formulas.
\end{theorem}

Before treating the case of non-square discriminant (case B) we
rephrase the setting of case A since this rephrasing facilitates the
proof in case B:

Firstly, as has been pointed out to us by Hiroshi Saito
\cite{saito_letter}, our results from \cite{bs_walds,bs_km} (and their
extension above) can be viewed using Hijikata's theory of 
embeddings of quadratic orders into quaternion orders
\cite{hijikata}. For this, let  $-d$ as above be a negative
discriminant and $f=X^2-sX+n\in \Z[X]$ with $s^2-4n=-d$. Then by
associating to $x\in L_i$ with $n(x)=d$ the linear map
$\varphi:k:=k_f:=\Q[X]/(f)\to D$ given by $\varphi(1)=1,
\varphi(X+(f))=\frac{x+s}{2}$, one obtains a one to one correspondence
between the $x \in L_i$ with $n(x)=d$ and the ring homomorphisms
$\varphi:k_f \to D$ with $\varphi(X+(f))\in R_i$. In this
correspondence, primitive vectors $x$ of norm $d$ 
correspond to optimal embeddings of the quadratic order of
discriminant $-d$ in $k_f$ into $R_i$.

Given two such embeddings $\varphi_1,\varphi_2$ there exists  $z\in
D^\times$ with $z^{-1}\varphi_1 z=\varphi_2$, and the space 
$$U(\varphi_1,\varphi_2)=\varphi_1(k_f)z=z\varphi_2(k_f)$$ is
independent of the choice of $z$. Moreover, $\overline{\varphi_2}$
given by $\overline{\varphi_2}(X+(f))=\frac{s-x}{2}$ has the same
image $\varphi_2(k_f)=\overline{\varphi_2}(k_f)$ as
$\varphi_2$, and the spaces $U_1,U_2$ described above are (with some
numbering) the same as $U(\varphi_1,\varphi_2),
U(\varphi_1,\overline{\varphi_2})$.
In addition, we have
$U(\varphi_1,\varphi_2)=U(\overline{\varphi_1},\overline{\varphi_2}),
U(\varphi_1,\overline{\varphi_2})= U(\overline{\varphi_1},\varphi_2)$,
so that we obtain a two to one correspondence between pairs  of
embeddings of orders in $k_f$ into quaternion orders $R_i,R_j$, at least one of which is optimal, 
on one side and primitive binary sublattices $M$ of $I_{ij}$ on the
other side. Our arguments from \cite{bs_walds,bs_km} can then be replaced  
using results from \cite{hijikata}.

To combine this setup with the results of Ponomarev \cite{pono_arith}
for the treatment of case B we notice that in case A the second Clifford
algebra $C^+(D)$ of the quadratic space $(D,n)$ can be identified with
$D\oplus D\cong D\otimes K$ with $K=\Q\oplus \Q$ in the way described
in \cite{pono_arith} and that the quadratic space $(D,n)$ is then
retrieved as being isometric to 
\begin{equation*}
 \tilde{D}:=\{\alpha \in C^+(D) \mid \bar{\alpha}^\tau=\alpha\}, 
\end{equation*}
where $\tau$ is the involution on $C^+(D)=D\otimes K$ induced by the
automorphism $(x_1,x_2)\mapsto (x_2,x_1)$ of $K=\Q\oplus \Q$, and
where $\alpha\mapsto \bar{\alpha}$ is the extension of the
quaternionic conjugation on $D$ to $D\otimes K$ which is trivial on
$1\otimes K$.

We can then view a pair $(\varphi_1,\varphi_2)$ of embeddings of $k_f$
into $D$ as giving an embedding $\varphi: y\mapsto
(\varphi_1(y),\varphi_2(y))\in C^+(D)$ of $k_f$ into $C^+(D)$. 
If $\varphi_1,\varphi_2$ map the order $\fo_f\subseteq k_f$ of
discriminant $-d$ into $R_i,R_j$ respectively,
$\varphi=(\varphi_1,\varphi_2)$ maps $\fo_f$  into the order
$R_i\oplus R_j$ of $D\oplus D \cong C^+(D)$. This embedding is optimal
in the sense that $\varphi(k_f)\cap(R_i\oplus R_j)=\varphi(\fo_f)$ if
and only if at least one of $\varphi_1,\varphi_2$ is an optimal
embedding of $\fo_f$ into $R_i$ resp. $R_j$.

Taking,
as above, $z \in D^\times$ with $z^{-1}\varphi_1z=\varphi_2$ and
writing $\alpha=(z,\bar{z})\in C^+(D)$, we have
$\bar{\alpha}^\tau=\alpha$,
$$\alpha^{-1}(\varphi_1,\overline{\varphi_2})\alpha=
(\varphi_2,\overline{\varphi_1})=
\overline{(\varphi_1,\overline{\varphi_2})}^\tau$$ and 
\begin{equation*}
(\varphi_1,\overline{\varphi_2})\alpha=(\varphi_1z,\overline{\varphi_2}\bar{z})
=(\varphi_1z,\overline{z\varphi_2})=(\varphi_1z,\overline{\varphi_1z}),  
\end{equation*}
so that
\begin{equation*}
\{(\varphi_1(y),\overline{\varphi_2(y)})\alpha\mid y \in
k_f\}\subseteq C^+(D)  
\end{equation*}
is the image of the two-dimensional space
$U(\varphi_1,\varphi_2)\subseteq D$ from above under the
identification of $D$ with $\tilde{D}\subseteq C^+(D)$.

The results from \cite{bs_walds,bs_km} and their reinterpretation in Saito's
approach can therefore also be viewed  as a computation of
discriminants in the two to one correspondence between 
optimal embeddings of quadratic orders in $k_f$ into the order  $R_i\oplus
R_j=(y_i,y_j)(R\oplus R)(y_i,y_j)^{-1}\subseteq C^+(D)$ and primitive binary sublattices of
$I_{ij}=\tilde{D}\cap(y_i,y_j)(R\oplus R)\tau((y_i,y_j)^{-1})\subseteq \tilde{D}\subseteq C^+(D)$.

To generalize this to the non split case we need a variant of the
Skolem-Noether theorem:

\begin{theorem}\label{skolemnoether}
Let $B$ be a central simple algebra over the field $K$ of odd
characteristic with an involution $x \mapsto x^I$ of the second kind,
denote by $K_0$ the fixed field of $I$.

Let $C_1$ be a commutative $K_0$-algebra with an embedding (of
$K_0$-algebras) $\varphi:C_1 \to
B$ and assume that $\varphi(C_1)$ and $K$ are linearly disjoint
over $K_0$, i. e., the $K$-algebra $\widetilde{\varphi(C_1)}$ generated by 
$\varphi(C_1)$ is isomorphic to $K\otimes_{K_0}C_1$.  Denote by $C_2$
the centralizer (commutant) in $B$ of 
$\widetilde{\varphi(C_1)}$.

Then  the
set of $\alpha \in B$ such that $\alpha \varphi^I =\varphi \alpha$
and $\alpha^I=\alpha$ is a $K_0$-vector space of dimension
$\dim_{K}(C_2)$.

If $B$ is a division algebra there exists $\alpha \in
B$ with $\alpha^I=\alpha$ such that $\alpha^{-1} \varphi
\alpha=\varphi^I$.
\end{theorem}

\begin{proof}
By the assumption of linear disjointness we can continue the
$K_0$-iso\-mor\-phism $\varphi(x)\mapsto \varphi(x)^I$ from $\varphi(C_1)$
to $\varphi(C_1)^I$ to a $K$-algebra isomomorphism $\rho$ from
$\widetilde{\varphi(C_1)}$ to the $K$-algebra generated by $\varphi(C_1)^I$.

By the theorem of Skolem-Noether there exists $\alpha_1\in B$ with
$\alpha_1^{-1} b \alpha_1= \rho(b)$ for all $b
\in\widetilde{\varphi(C_1)}$, in particular we have $\alpha_1^{-1}
\varphi(x) \alpha_1= \varphi(x)^I$ for all $x \in C_1$. 

The set of $  \alpha \in B$ such that $\alpha \varphi(x)^I =\varphi(x)
\alpha$ holds for all $x \in C_1$ is then a non zero $K$-vector space $W_0$ of dimension
$\dim_K(C_2)$.  Moreover, for $\alpha\in W_0$ application of $I$ gives
$\alpha^I\varphi^I=\varphi\alpha^I$, so that $I$ operates on
$W_0$, viewed as a vector space over $K_0 \subseteq K$. Multiplication by a $c \in K$ with $c^I=-c$ switches the $+1$
and the $-1$ eigenspace of $I$ in $W_0$, so both eigenspaces have the
same dimension over $K_0$, namely $\dim_K(C_2)$. 
If $B$ is a division algebra, all nonzero $\alpha \in
W_0$ are invertible and satisfy $\alpha^{-1} \varphi \alpha=\varphi^I$.
\end{proof}
\begin{myrem}
One could also use recent results of Villa \cite{villa} for the proof
of the above theorem.  
\end{myrem}
\begin{corollary}\label{space_correspondence}
Let $D$ be a quaternion algebra over $\Q$ with standard involution $b
\mapsto \bar{b}$, let $K$ be a quadratic extension of $\Q$ and put
$D_K=D\otimes_\Q K$, let $\tau$ be the nontrivial automorphism of $K$.
Denote by $I$ the involution $b \mapsto \tau(\bar{b})$ of the
second kind of $D_K$, by $V$ the ($4$-dimensional) $\Q$ vector space of all $I$-invariant
elements of $D_K$ and equip $D_K$ and $V$ with the quaternionic norm
form of $D_K$ as a quadratic form. Let $k=\Q(\sqrt{-d})$ be a quadratic extension of $\Q$ of
discriminant $-d$ which can be embedded into $D_K$ and is linearly
disjoint from $K$. 

Then there is a bijective  correspondence between (unordered)  pairs of embeddings 
$\varphi_1,\varphi_2:k \to D_K$ with
$\varphi_2(x)=\overline{\varphi_1(x)}$ for all $x \in k$  and binary
quadratic subspaces $W$ of determinant $d$ of $V$.
This correspondence is given by associating to $\varphi\in\{\varphi_1,\varphi_2\}$ as above the
set $W$ of $\alpha \in V$ with $\varphi(x)\alpha=\alpha \varphi(x)^I$ for
all $x \in k$.
\end{corollary}
\begin{proof}
Given an embedding $\varphi$, the set $W\subseteq V$ is a
$2$-dimensional $\Q$ vector space by the Theorem, and if we have
$\alpha \in W$ invertible, one sees that $W=\{\varphi(x)\alpha\mid x
\in k\}$, so that the quadratic space $W$ has the same determinant
$d$ as the quadratic space $k$ (equipped with the norm form).    
If $\varphi_1, \varphi_2$ are as above the associated sets of $\alpha
\in V$ are obviously the same. Conversely, given a binary subspace $W$ of
determinant $d$ of $V$, we consider an orthogonal basis $\{w_1,w_2\}$
and put $b:=w_1\overline{w_2}$, we have $\tr(b)=0$ and
$n(b)=d$. Associating to $W$ the embedding $\varphi:k=\Q(\sqrt{-d})\to
D_K$ with $\varphi(\sqrt{-d})=b$ we have constructed the inverse of  the map
$\varphi \to W$ defined above. 
\end{proof}
\begin{corollary}\label{lattice_correspondence}
With notations as above we  let (as in the discussion of case B in
Section 2) $R$ be a symmetric order in $D_K$ and
$\Lambda_i=y_iR\tau(y_i)^{-1}\cap V$ be
a lattice in the genus of $\Lambda_1:=R\cap V$, put $R_i:=y_iRy_i^{-1}$.
Then we have a bijection between pairs $(\varphi,
\bar{\varphi})$ of optimal embeddings of orders of (not necessarily
fundamental) discriminant
$-d$ in $k=\Q(\sqrt{-d})$ into $R_i$ and of primitive binary sublattices of determinant
$dm^2$ of $\Lambda_i$ 
with $m \in \Z$.  
\end{corollary}
\begin{proof}
The assertion follows directly from the previous corollary by taking
the preimage of $\varphi(k)\cap R_i$ as order in $k$ and $\Lambda_i \cap W$ as
binary sublattice in $\Lambda_i$.  
\end{proof}
As in case A it remains to compute the discriminants in the matching
of the last corollary. For this we restrict now to the case that $2$
is unramified in $K/\Q$ and that the primes ramifying in $D$ are
precisely those that are ramified in $K/\Q$. This implies that the
discriminant of $\Lambda$ is odd
and square free (congruent to $1$ modulo $4$). A more general
situation (as in case A) 
can probably be handled as well, but the matching
leads then to quite complicated and unpleasant formulas, since we have
to handle determinants $d, ds^2, d\Delta,ds^2\Delta$ simultaneously,
where $s$ runs over divisors of the level.

\begin{lemma}
Let $D$ be a  quaternion algebra over $\Q$, let $\Delta$ be an odd
square free determinant, $K=\Q(\sqrt{\Delta}), D_K=D\otimes_\Q K$,
assume that the primes ramifying in $D$ are
precisely those that are ramified in $K/\Q$.

Let $R$ be a symmetric maximal order in $D_K$ and $\Lambda=R\cap V$,
where $V=\{\alpha \in D_K \mid \bar{\alpha}^{\tau}=\alpha\}$.

With a double coset decomposition $D_{K,\A}^1=\cup_{i=1}^hD_K^1
y_i \R_\A^1$ (with $y_1=1$ and $n(y_i)=1$ for all $i$ ) put  $I_i:=y_i
R \tau(y_i)^{-1}, \Lambda_i=V\cap  I_i$ and $R_i=y_iR y_i^{-1}$.

Let $k$ be a
quadratic field different from $K$ which can be embedded into $D_K$
and
$\varphi: {\mathfrak o}\to R_i$ an optimal embedding of the order
${\mathfrak o}\subseteq k$ of discriminant $-d$ of $k$.

Then the primitive binary sublattice $M\subseteq V$ of $\Lambda_i$ associated to
$\varphi$ by Corollary \ref{lattice_correspondence} has determinant $d$. 
\end{lemma}
\begin{proof}

As in case A we check this locally, writing  the index $p$ indicating
completion only where necessary for clarification in the sequel. Also,
for this it is enough to consider $\Lambda=\Lambda_1$.

If $p$ is a prime that splits in $K/\Q$ we are locally in the
situation of case A and can apply the result proven there.

If $p$ is inert in $K/\Q$ the algebra $D\otimes K_p$ is isomorphic to
the matrix ring over $K_p$ and we can, as in \cite{bs_walds,bs_km}, find an
$\alpha_1 \in R_p^\times$ with $\alpha_1^{-1} \varphi(x)
\alpha_1=\tau(\overline{\varphi(x)})$ for all $x \in k$ (the argument
given in \cite{bs_walds,bs_km} is  for the matrix ring over $\Z_p$ but stays
valid if $\Z_p$ is replaced by the integers of any local field).
Since $K_p/\Q_p$ is unramified, Hilbert's Theorem 90 holds for the
local units, and similarly as in \cite[Section 2.5]{kneser_galois} we can modify
$\alpha$ by a unit $\epsilon \in {\mathfrak o}_K^\times$ to obtain
$\alpha\in R_p^\times\cap V$ with
$\alpha^{-1}\varphi(x)\alpha=\tau(\overline{\varphi(x)})$ for all $x
\in k$. Since $M=\varphi({\mathfrak o}\alpha)$ by our construction, $M$
has the same determinant $d$ as ${\mathfrak o}$ (with the norm form as
quadratic form).

\medskip
If $p \ne 2$ is ramified in $K/\Q$, we have by assumption that $p$ is
ramified in $D$. We denote by $\hat{R}=\hat{R}_p=\{z \in D_p \mid
n(z)\in \Z_p\}$ the unique maximal order in 
$D_p$ and choose $R=R_p=\hat{R}\otimes {\mathfrak o}_K+D^{(p)}\otimes
\delta^{-1}{\mathfrak o}_K$ with
$\delta=\sqrt{\Delta} \in K$ and $D^{(p)}=\{z\in D_p\mid n(z)\in
p\Z_p\}$; it is easily verified that $R$ is indeed a symmetric maximal
order in (the completion of) $D_K$ and that $R\cap D=\hat{R}$ holds (identifying $D$ with
$D\otimes 1 \subseteq D_K$).

As above we write  $k=k_f=\Q[X]/(f)$ with
$f=X^2-sX+n\in \Z[X]$ and $s^2-4n=-d$ and $ {\mathfrak o}={\mathfrak o}_f=\Z[X]/(f)$.
To a vector $w \in L_p=\{z \in \Z_p1+2R_p=R_p\mid
{\rm tr}(z)=0\}$ of norm $d$ we associate the embedding $\varphi$ of
${\mathfrak o}$ into $R$ given by $2\varphi(X+(f))-s=w$ and vice
versa. As in 
\cite{bs_walds,bs_km} by this correspondence between  the embedding
$\varphi$ and the
vector  $w$ we obtain a bijection between optimal embeddings of
${\mathfrak o}$ into $R$ and $\Z$-primitive vectors $w\in L$ (i.e.,
$rw \in L$ with $r\in \Q$ implies $r\in \Z$) with $n(w)=d$.

The vector $w$ can be written as $v_1\otimes 1+v_2\otimes \delta$ with
$v_1,v_2 \in D$ of trace $0$, we have
$\tau(\bar{w})=-v_1\otimes 1+v_2\otimes \delta$ and $d=n(v_1)+\Delta n(v_2)$. The
condition $n(w)=d\in \Z$ then implies that $v_1,v_2$ are orthogonal
with respect to the trace bilinear form on $D$. 

The vector
$\alpha:=rv_2\otimes \delta+tv_1v_2\otimes \delta+tv_2^2
\Delta$ with $r,t \in \Q$ satisfies then $\tau(\bar{\alpha})=\alpha$
and $\alpha \tau(\bar{w})=w\alpha$ and hence
$\alpha^{-1}\varphi(x)\alpha=\tau(\overline{\varphi(x)})$ for all $x
\in k$.

We notice that in the trace zero part $\hat{R}^0$ of $\hat{R}$ each vector of norm
in $\Z_
p^\times$ can be split off orthogonally with $p$-modular
orthogonal complement, in particular, all vectors in $\hat{R}^0$
orthogonal to it have norm divisible by $p$. 

We have $$n(\alpha)=\Delta n(v_2)(r^2+t^2n(v_1)+t^2\Delta
n(v_2))=\Delta n(v_2)(r^2+t^2d).$$ 
Since $R=R_p$ is symmetric we have
$v_1\otimes 1 \in R, v_2\otimes\delta \in R$.  This implies $v_1 \in \hat{R}, v_2 \in
p^{-1}\hat{R}$ with $pn(v_2)\in \Z$, and at least one of $v_1,pv_2$ is
primitive in $\hat{R}$. In particular, since $\hat{R}$ consists of all
elements of $D$ of integral norm, we have $n(v_1)\in \Z^\times_p \cup
p\Z^\times_p$ or
$n(v_2)\in p^{-2}\Z^\times_p\cup p^{-1}\Z^\times_p$. In the second
case we see that in
fact $n(v_2)\in p^{-1}\Z^\times_p$ must hold. We can then choose $r=1,
t=p$ above and obtain $\alpha \in R$ with $n(\alpha)\in \Z_p^\times$, hence $\alpha \in
R^\times$, which proves the assertion in this case. 

If
$pv_2=p^\nu\tilde{v_2}$ with $\nu>0$, $\tilde{v_2}$ primitive in
$\hat{R}$, we write $t=p^{-\nu}\tilde{t}$ and have $\Delta
n(v_2)t^2=p^{-1}\tilde{t}^2n(\tilde{v_2})$ and
$tv_2=p^{-1}\tilde{t}\tilde{v_2}$.
If $d \in \Z_p^\times$ one has $n(v_1)\in \Z_p^\times$ and therefore
$n(\tilde{v_2})\in p\Z_p^\times, n(v_1\tilde{v_2})\in p\Z_p^\times$,
which implies $p^{-1}\tilde{v_2}\otimes \delta \in R, p^{-1}v_1
\tilde{v_2}\otimes \delta \in R$  and hence $\alpha \in R$ for  $r \in
p^{-\nu}\Z_p, \tilde{t}\in \Z_p$. With $\tilde{t}=p, r=p^{-\nu}$ we
obtain $n(\alpha)\in \Z_p^\times$ since $\Delta n(v_2)r^2\in
\Z_p^\times$ and $\Delta n(v_2)t^2d \in p\Z_p^\times$.

If $d \in p\Z_p^\times$ we have $n(v_1)\in p\Z_p^\times$ and hence
$p^{-1}v_1\tilde{v_2}\otimes \delta \in R$ if $n(\tilde{v_2})\in
\Z_p^\times$ and $p^{-1}\tilde{v_2}\otimes \delta \in R,
p^{-2}v_1\tilde{v_2}\otimes \delta \in R$ if $n(\tilde{v_2})\in
p\Z_p^{\times}$. In the first case we put $\tilde{t}=1, r=1$, in the
second case we put $\tilde{t}=1, r=p^{-\nu}$. In both cases we obtain
$\alpha \in R$ with $n(\alpha)\in \Z_p^\times$, which proves the
assertion.

If finally $d \in p^2\Z_p$, we must have $n(v_1)\in p\Z_p^\times,
n(v_2)\in \Z_p^\times, \tilde{v_2}=v_2$. We can then take
$\alpha=v_2\otimes\delta$ with $n(\alpha)=n(v_2)\Delta \in
pZ_p^\times$.
This gives the $2$-dimensional subspace generated by
$v_2\otimes\delta, -\Delta n(v_2)+v_1v_2\otimes\delta$ of $V$; its
intersection with $\Lambda$ is the lattice
$\Z v_2\otimes\delta+\Z(-n(v_2)+\Delta^{-1}v_1v_2\otimes\delta)$, with
imprimitive Gram matrix (divisible by $p$ once) of determinant
$n(v_2)^2 d\in d (\Z_p^\times)^2$.
\end{proof}
\begin{myrem}
The reader may wonder (at least the author did) why certain types of
possible binary sublattices of $\Lambda$ do not appear in the
correspondence. One checks easily (by reducing modulo $p$) that indeed
at primes $p$ which are inert in $K/\Q$ the lattice 
$\Lambda_p$ has no primitive binary sublattices with $p$-imprimitive
Gram matrix (because its reduction mod $p$ has Witt index $1$).
\end{myrem}
\begin{theorem} \label{averagetheorem_nonsplit}
Let $\Delta$ be an odd
square free discriminant with an odd number of prime factors 
and $V$ the (unique up to isometry) positive
definite quadratic space of determinant $\Delta$ which has Witt
invariant $-1$ at all the primes dividing $\Delta$.
Realize $V$ by letting $D$ be the definite quaternion algebra over
$\Q$ ramified at all primes dividing $\Delta$ and   $K=\Q(\sqrt{\Delta}), D_K=D\otimes_\Q K$,
$V=\{\alpha \in D_K \mid \bar{\alpha}^{\tau}=\alpha\}$, with the
quaternionic norm form as quadratic form.

Let $\Lambda=R\cap
V$, where $R$ is a symmetric maximal order in $D_K$, be a maximal
lattice (of determinant $\Delta$) on $V$, put $L=\{z \in \Z1 +2R\mid
\tr(x)=0\}$. Let $\Lambda_i=y_iR\tau(y_i^{-1})$ with $y_i \in
D_{K,\A}^\times, n(y_i)\in \Q_\A^\times$ be a lattice in the genus of $\Lambda$,
put $L_i=y_iLy_i^{-1}$. 

Then for all $d \in \N$ one has 
\begin{equation*}
 r_{av}^*(\Lambda_i, d)=r^*(L_i,d), 
\end{equation*}
where $r^*(L_i,d)$ denotes the number of $\Z$-primitive representations
of $d$ by $L_i$.
\end{theorem}
\begin{proof}
This follows from the previous Lemma.  
\end{proof}
\begin{myrem}
\begin{enumerate}
\item The assertion in case B looks smoother than the result of Theorem
\ref{splitcase_average}. As explained earlier this is due to the fact
that we have restricted our attention to the maximal lattices $\Lambda$  on the
particular type of quadratic space $V$ described above. It appears
that the formulas for more general $\Lambda$ of non square determinant
will become messier than in the split case.   

\parindent=0pt
At least for $\Lambda$ of prime determinant our case is the general one.
\item As in \cite{bs_walds,bs_km} the formula for primitive
  representations above implies the relation
  \begin{equation*}
 r(L,d)=\sum_{m^2\mid d}\mu(m)m r_{av}(\Lambda,\frac{d}{m^2})   
  \end{equation*}
for representation numbers without primitivity condition.
From this one can obtain a formula relating the Koecher-Maaß series of a
type II Yoshida lift (as described in \cite{yoshida}) with the Mellin
transform of a corresponding Hilbert modular form of weight $3/2$,
replacing Corollary 2.2 of \cite{bs_km} in the non-split 
situation of case B.
Moreover, as in \cite{bs_km}, one can then obtain a proof of Böcherer's
conjecture for the Koecher-Maaß series of these Yoshida liftings of
type II. Details of this will be worked out separately.
\end{enumerate}
\end{myrem}
\section{Average asymptotic formula and applications}
We keep the notations of the preceding section for the cases A, B
respectively.
In particular, in case A we consider 
an order $R$ in the definite quaternion algebra $D$ over $\Q$ such that the completion $R_p$ is
\begin{itemize}
\item an order of level $p^{r_p}$ with $r_p$ odd (see
  \cite{pizer_quaternion2}) for the primes $p$ at which $D$ is
  ramified
\item an Eichler order of level $p^{r_p}$ for the primes $p$ at which
  $D$ splits,
\end{itemize}
with $r_p=0$ for almost all $p$ and write $N_1$ for the product of the
$p^{r_p}$ for the finite primes $p$ at which $D$ ramifies, $N_2$ for the product of the
local levels $p^{r_p}$ for the primes $p\nmid N_1$ and $N=N_1N_2$. The
lattice $\Lambda$ of determinant $N^2$ and level $N$ considered in case A is then one of the ideals
$I_{ij}$ introduced in Section 2, with associated ternary $\Z$-lattices
$L=L_i, L'=L_j$ as described in Section 2.
In case B we consider maximal $\Z$-lattices $\Lambda$
of odd square free determinant $\Delta\equiv 1 \bmod 4$, which are
realized in the quadratic space $V=\{\alpha \in D_K=D\otimes K\mid
\tau(\alpha)=\bar{\alpha}\}$ over $\Q$, where $K=\Q(\sqrt{\Delta})$
and  $D$ is a definite quaternion algebra over $\Q$ ramified precisely
at the primes ramified in $K/\Q$, i.e, at those dividing $\Delta$. To
such a $\Lambda$ we have associated a ternary ${\mathfrak
  o}_K$-lattice $L$ in the trace zero part of $D_K$.
\begin{lemma}\label{ternary_asymptotic}
With notations as above one has
\begin{eqnarray*}
r(L,d)r(L',d)&=&r(\gen(L),d)^2+\bigoh(d^{1-\frac{1}{16}+\epsilon}),\\
r^*(L,d)r^*(L',d)&=&r^*(\gen(L),d)^2+\bigoh(d^{1-\frac{1}{16}+\epsilon})
\quad\quad\ \text{in case A}\\
r^*(L,d)&=&r^*(\spn(L),d)+\bigoh(d^{1-\frac{25}{256}+\epsilon})
\quad\quad \text{in case B}, 
\end{eqnarray*}
where $r(\gen(L),d)$ respectively $r^*(\gen(L),d)$ denotes Siegel's weighted average of the
numbers of representations respectively $\Z$-primitive representations of $d$ by the isometry classes of
lattices in the genus of $L$ and where $r^*(\spn(L),d)$ denotes the
analogous average over the isometry classes in the spinor genus of
$L$.

Moreover, one has $\gen(L)=\spn(L)$ in case A and
$$r(\spn(L),d)=r(\gen(L),d)\cdot
\begin{cases}
  2\\1
\end{cases},$$ depending on
whether the square class of $d$ is or is not spinor exceptional in the
sense of \cite{kneser_mz}.
\end{lemma}
\begin{proof}
In both cases one sees from
\cite{kneser_klassenzahlen,earnest_hsia_spinornorms2} that the local
spinornorms of the lattice are all square classes of units at all
finite places. In case A this implies that the genus of $L,L'$
consists of only one spinor genus and the assertion follows from
\cite{sp_inventiones,duke_sp} and the currently best estimate for Fourier
coefficients of cusp forms of weight $3/2$ (orthogonal to the
one-dimensional theta series) of \cite{blomer_harcos_crelle}.

In case B the genus of $L$ consists of more than one spinor genus if
the class number of $K$ is even, with all the lattices $L_i$ arising
from the orders $R_i=y_iRy_i^{-1}$ with $n(y_i)\in \Q_\A^\times$ in
the same spinor genus. 
From \cite{earnest_hsia_hung} one
reads off that $d$ is a primitive spinor exception if and only if the square class
of $d$ is spinor exceptional in the sense of \cite{kneser_mz} (this case
occurs if and only if $\Q(\sqrt{-d\Delta},\sqrt{\Delta})$ is unramified over
$\Q(\sqrt{\Delta})$). In this case by \cite{kneser_mz} the genus of $L$
splits into two halves (``half genera'') consisting of equally many
spinor genera and such that $r(\spn(L'),d)$ resp. $r^*(\spn(L'),d)$ has the
same value for all lattices $L'$ in the same half, in particular it is zero for the
lattices in one
of the half genera and equal to $2r(\gen(L),d))$ resp. $2r^*(\gen(L),d))$
for lattices in the other half genus.

It is easily checked that for all $d$
represented locally everywhere by $L$ there are binary lattices of
determinant $d$ which are represented locally everywhere primitively
by $\Lambda$, 
hence are represented by some lattice $y\Lambda\tau(y^{-1})$ with
$n(y)\in \Q_\A^\times$ in the
genus of $\Lambda$. Theorem \ref{averagetheorem_nonsplit} then implies
that $d$ is represented primitively by the ternary lattice $yLy^{-1}$, which is in
the spinor genus of $L$. From the argument above we obtain $2
r^*(\gen(L),d)=r^*(\spn(L),d)$ in the case that the square class of
$d$ is spinor exceptional. On the other hand,
$r^*(\gen(L,d))=r^*(\spn(L,d))$ follows from \cite{kneser_mz} if the
square class of $d$ is not exceptional.

By \cite{SP-nagoya}
$r^*(L,d)-r^*(\spn(L),d)$ is the $\Z$-primitive Fourier coefficient
$$a_f^*(d)=\sum_{m \in \N, m^2\mid d}\mu(m)a_f(d/m^2)$$ at $d$ of a
Hilbert cusp form $f$ of weight $3/2$ in the orthogonal complement of the
unary theta series with Fourier coefficients $a_f(d)$. 

The estimate of
\cite{blomer_harcos_gafa,harcos_hilbertmodular} for the $a_f(d)$ then proves the
assertion. Notice that the restrictive condition on the represented
number of \cite[Theorem 2]{harcos_hilbertmodular} can be dropped here
since we have established that $d$ is either not in a spinor
exceptional square class or is primitive spinor exceptional
represented by the spinor genus of $L$. 
\end{proof}
\begin{theorem}\label{average_asymptoticformula} Let $-d$ run through discriminants such that binary lattices
  of determinant $d$ are represented primitively by the genus of
  $\Lambda$, and assume in case A that $d$ is not divisible by any
  prime $p\mid N$ such that $r_p>1$, write $N_d:=\frac{N}{\gcd(N,d)}$.
 Let $\epsilon>0$. 

Then there are constants $C_i=C_i(\epsilon, N)$
  for $i=1,2$  such
  that: 

\begin{enumerate}
 \item In case A for  $N=p^r$ a prime power one has 
   \begin{equation*}
 r_{av}^*(\Lambda,d)=
r_{av}^*(\gen(\Lambda),d)+\bigoh(d^{1-\frac{1}{16}+\epsilon})  
   \end{equation*}
and 
\begin{equation*}
  r_{av}(\Lambda,d)
  \ge C_2 d^{1-\epsilon} \text{ for } d>C_1 .
\end{equation*}
\item In case A with $N$ having more than one prime factor let
$N'\| N$ 
be  a  fixed divisor of $N$ dividing $N$ exactly and consider $d$ as
above with $N_d=\frac{N}{\gcd(N,d)}=N'$. Then one has  
\begin{equation*}
 r_{av}^*(\rpd(\Lambda,N'),d)= \sigma_0(N') r_{av}^*(\gen(\Lambda),d)+  \bigoh(d^{1-\frac{1}{16}+\epsilon})
\end{equation*}
and
\begin{equation*}
  r_{av}(\rpd(\Lambda,N'),d)
  \ge C_2 d^{1-\epsilon} \text{ for }  d>C_1  \text{ with } N_d=N'. 
\end{equation*}
Here we write
$r_{av}^*(\rpd(\Lambda,N'),d)=\sum_{s\mid N'}r_{av}^*(\Lambda^{*,s},d)$ as
in Section 2 and analogously for $r_{av}(\rpd(\Lambda,N'),d)$.
\item  In case B, in the situation of Theorem
  \ref{averagetheorem_nonsplit}, one has 
  \begin{equation*}
 r^*_{av}(\Lambda,d)=r^*_{av}(\spn(L),d)+\bigoh(d^{1-\frac{25}{256}+\epsilon})
  \end{equation*}
and
\begin{equation*}
  r_{av}(\Lambda,d) \ge C_2
  d^{1-\epsilon} \text{ for } d>C_1.
\end{equation*}
 \end{enumerate}
If one restricts here to $d$ for which $-d$ is fundamental  the local
condition is satisfied in case A if and only if $p$ is not split in
 $\Q(\sqrt{-d})$ for all $p \mid N_1$ and is not inert in this quadratic
extension for all $p \mid N_2$. 
It is satisfied in case B for all
positive $d$ for which $-d$ is a discriminant.
\end{theorem}
\begin{proof} 
For assertion b) we write $e_i=\vert R_i^\times\vert$ and have 
\begin{eqnarray*}
\sigma_0(N')r_{av}^*(\gen(\Lambda),d)&=&\sum_{i,j}\sum_{s\|N'}\frac{r_{av}^*(I_{ij}^{*,s},d)}{e_ie_j}\\
&=&\sum_{i,j}\sum_{m\in \N,m^2\mid
    d}\mu(m)\frac{r(L_i,\frac{d}{m^2})}{e_i}\frac{r(L_j,\frac{d}{m^2})}{e_j}\\
&=&\sum_{m\in\N, m^2\mid d}\mu(m)(r(\gen(L),\frac{d}{m^2})^2.  
\end{eqnarray*}
Since we have 
\begin{equation*}
r_{av}^*(\rpd(\Lambda,N'),d)=\sum_{m\in\N,m^2\mid d}\mu(m)r(L,\frac{d}{m^2})r(L',\frac{d}{m^2})  
\end{equation*}
by equation \ref{kmformula1b}, the asymptotic formula follows from the previous
lemma. The lower bound given for $r_{av}(\Lambda,d)$ then follows from
\begin{eqnarray*}
 \sigma_0(N_d) r_{av}(\gen(\Lambda),d)&\ge&\sigma_0(N_d) r_{av}^*(\gen(\Lambda),d)\\
&\ge & \sum_{i,j}\frac{r^*(L_i,d)r^*(L_j,d)}{e_ie_j}\\
&=&(r^*(\gen(L),d))^2\\
&\ge & cd^{1-\epsilon}
\end{eqnarray*}
for some constant $c>0$ if $d$ is large enough and represented locally
everywhere.

The other cases are proved similarly; in case c) we make use of the
fact that  $r^*(\spn(L),d)$ is equal to either $r^*(\gen(L),d)$ or $2
r^*(\gen(L),d)$ by Lemma \ref{ternary_asymptotic} if $d$ is as assumed.
\end{proof}
\begin{myrem}
The appearance of the set of all rescaled partial duals $\Lambda^{*,s}$
with $s \mid N'$ of $\Lambda$ in the
case of composite $N$ is somewhat unsatisfactory. The number of lattices
involved is obviously $2^{\omega(N')}$, but the set of distinct
isometry classes of these lattices may be smaller. The argument from
\cite{bs_km} would allow to replace representations of $T$ by
$\Lambda^{*,s}$ by representations of $sT$ by $\Lambda$ itself, if one
likes that better.
\end{myrem}

We restrict attention now to fundamental discriminants $-d$. 
\begin{proposition}
Let $-d$ be a fundamental discriminant as in Theorem
\ref{average_asymptoticformula} and $N_d=\frac{N}{\gcd(N,d)}$. Put
\begin{equation*}
  w=
  \begin{cases}
    6 & d=3\\
4&d=4\\
2& \text{otherwise}
  \end{cases}
\end{equation*}
and denote by $h(-d)$ the number of $SL_2(\Z)$-equivalence classes
of binary quadratic forms of discriminant $-d$ (equal to the ideal
class number of $\Q(\sqrt{-d})$) and by $2^t$ the number of genera of
such forms.

Then in case A we have:
\begin{equation}\label{massformel}
  h(-d)\sigma_0(N_d)r(\gen(\Lambda),T)=w(r(\gen(L),d)^2,
\end{equation}
where $\sigma_0$ denotes the number of divisors function and $T$ is an
integral binary symmetric matrix of discriminant $-d$.

In case B we have
\begin{equation}
  \label{massfomel_caseB}
\frac{h(-d)}{2^t}\sum_{T_i}r(\gen(\Lambda, T_i))=wr(\spn(L),d),  
\end{equation}
where the $T_i$ run over a set of representatives of the genera of
discriminant $-d$.
\end{proposition}
\begin{proof}
In case A Siegel's weighted average $r(\gen(\Lambda),T)$ of the
numbers of representations of $T$ by the classes of quadratic lattices
in the genus of $(\Lambda,n)$ is independent of $T$ since the local
lattices allow similitudes with an arbitrary unit as similitude
norm. In particular $r(\gen(\Lambda),T)$ has the same value 
for $T$ in all genera of binary quadratic forms of discriminant $-d$.
The assertion then follows as in the proof of Theorem
\ref{average_asymptoticformula}.

In case B such similitudes do not exist and $r(\gen(\Lambda), T)$ is
still the same for $T$ in the same genus but may
be different for  $T$ in different  genera of binary quadratic forms
of discriminant $-d$, in particular it may be zero for some and
nonzero for others. We have then
\begin{eqnarray*}
 wr(\gen(L),d)&=&wr_{av}(\gen(\Lambda,d)) \\
&=&\frac{h(-d)}{2^t}\sum_{T_i}r(\gen(\Lambda, T_i))
\end{eqnarray*}
as asserted.

Of course, the assertion could also be established by a direct computation of
local densities.
\end{proof}

We can now combine this last result with Theorem
\ref{average_asymptoticformula} 
and obtain  a surprising result about representations of
individual forms $T$ by $\Lambda$. For this we denote by
$\mu=\mu(\gen(\Lambda))$ the mass of the genus of 
the quadratic lattice $\Lambda$ and by $o_{\rm max} =o_{\rm max}
(\gen(\Lambda))$ the maximal order of 
the group of automorphisms of a lattice in the genus of $\Lambda$.
With this notation we have:
\begin{theorem}
  Let $-d$ run through
  fundamental discriminants satisfying the conditions of Theorem
  \ref{average_asymptoticformula}. 
In case A with $N$ having more than one prime factor let $N'$ be a
fixed exact divisor of $N$ and restrict further to  $d$ with $N_d=N'$.
Denote by $\nu_d$ 
  the number of binary quadratic forms $T$ of 
  discriminant $-d$ with $r(\rpd(\Lambda,N'),T)\ne 0$ in case A with
  $N$ not being a prime power, with $r(\Lambda,T)\ne 0$ in the other
  cases.

Then for
  all $\delta>0$ there is a constant $C_3=C_3(\delta)$ such that for all $d\ge
  C_3$ as above one has
   \begin{equation*}
    \nu_d \ge (1-\delta) \frac{\sigma_0(N')h(-d)}{\mu o_{\rm max}}.
   \end{equation*}
  in case A with $N$ not being a prime power,
  \begin{equation*}
       \nu_d \ge (1-\delta) \frac{h(-d)}{\mu o_{\rm max}}
  \end{equation*}
in case A with $N$ being a prime power, and 
\begin{equation*}
 \nu_d \ge (1-\delta)\frac{h(-d)}{2^t \mu o_{\rm max}} 
\end{equation*}
in case B.

In particular, for $d$ large enough a positive proportion of the classes of binary quadratic
forms of discriminant $-d$ is represented by at least one among the
rescaled partial dual lattices of $\Lambda$ comprising the set
$\rpd(\Lambda,N_d)$ in case A with $N$ not being a prime power, by
$\Lambda$ itself in the other cases.
\end{theorem}
\begin{proof}
 From Siegel's mass formula we see that we must have
 \begin{equation*}
  r(\Lambda,T) \le  o_{\rm max}\mu r(\gen(\Lambda),T)
 \end{equation*} in case A with $N$ a prime power and in case B,
 \begin{equation*}
   r(\rpd(\Lambda,N'),T) \le o_{\rm max}\mu r(\gen(\Lambda),T),
 \end{equation*} in case A with $N$ not a prime power,
since all terms in the weighted sum for $r(\gen(\Lambda),T)$ are non
negative.

On the other hand, as in the proposition we see in case A from the asymptotic
formula  of Theorem \ref{average_asymptoticformula} that the sum of these
terms over the classes of $T$ of discriminant $-d$ is asymptotic to
$h(-d)r(\gen(\Lambda, T))$ respectively to
$\sigma_0(N')h(-d)r(\gen(\Lambda, T))$, so we can estimate it for
sufficiently large $d$ from below by $(1-\delta)$ times this number.
The number $\nu_d$ of terms contributing to the sum must therefore be
at least as large as asserted.

In case B we have $$r(\gen(\Lambda, T))\le \frac{2^t w
  r(\gen(L),d)}{h(-d)}$$
by the proposition, and the same argument as above gives the assertion
in this case too.   
\end{proof}

\begin{myrem}
  \begin{enumerate}
  \item The ``sufficiently large'' in the theorem is not effective
    since our use of the asymptotic formula relies on the fact that
    $h(-d)$ is at least of the order of $d^{\frac{1}{2}-\epsilon}$ for
      all $\epsilon >0$, which estimate is well known to be
      ineffective.
\item Our method appears not to be suitable to guarantee the existence
  of $T$ which is represented by $\Lambda$ and lies in a given small
  (say of size $d^\delta$ with $0<\delta<\frac{1}{2}$) subset of the
  class group; such a result was obtained by Einsiedler, Michel,
  Lindenstrauss and Venkatesh by the ergodic method.
\item If the number of represented forms $T$ is as small as it can be
  the representation number $r(\Lambda,T)$ (or $r(\rpd(\Lambda,N'),T)$
  has to be rather large, i.\ e.\ at least of the order of magnitude of the
  number $r(\gen(R),T)$ from Siegel's mass formula (in fact, even
  larger).
\item In contrast to most other results about representation of
  quadratic forms of rank $>1$ by quadratic forms our result does not
  at all involve the minimum of the binary quadratic form $T$ to be
  represented, at least not explicitly.  
  \end{enumerate}
\end{myrem}
\section{Fourier coefficients of Siegel modular forms}
The results from \cite{bs_km} quoted above were used there to compute
averages of Fourier coefficients of special Siegel modular forms of
degree $2$, the Yoshida liftings of type I. These averages
appear there as the coefficients in the Dirichlet series of Koecher
and Maa\ss\  associated to such a Siegel modular form. An immediate
consequence is the following:

\begin{proposition}
Let $F\in M_k^{(2)}(\Gamma_0(N))$ be the Yoshida lifting of a pair
$f_1,f_2$ of primitive cuspidal elliptic Hecke eigenforms of weights
$k_1=2,k_2=2k-2$ and square free level $N$ and denote by \begin{equation*}
a(F,d)=\sum_T\frac{a(F,T)}{\epsilon(T)}
\end{equation*}
the $d-$th  coefficient of its Koecher-Maa\ss\  series,  
where $a(F,T)$ is the $T$-th Fourier coefficient of $F$, the sum is
over a set of $SL_2(\Z)$-equivalence classes of 
integral symmetric matrices
$T=\bigl(\begin{smallmatrix}2a&b\\b&2c\end{smallmatrix}\bigr)$ of
(signed) discriminant $b^2-4ac=-d$, and where $\epsilon (T)$ denotes
the number of proper automorphisms (or units) of $T$.

Then for any $\epsilon>0$ there is a constant $C$ such that one has
for all $d$
\begin{equation}
  a(F,d)\le C d^{\frac{k}{2}-\frac{1}{8}+\epsilon}.
\end{equation}
If above we replace the cusp form $f_1$ by the (appropriate)  Eisenstein
series of weight $2$ and level $N$  we obtain
 \begin{equation}
  a(F,d)\le C d^{\frac{k}{2}-\frac{1}{16}+\epsilon}.
\end{equation}
 \end{proposition}  
 \begin{myrem}
The exponent at $d$ in our estimate depends on what is known about
Fourier coefficients of modular forms in one variable of half integral weight. 
We discuss this dependence and possible improvements and compare our
result on the average Fourier coefficient $a(F,d)$ with 
known results and conjectures about the size of the individual Fourier
coefficients $a(F,T)$:
\begin{enumerate}
\item If one assumes the generalized Ramanujan
  conjecture for cusp forms of half integral weight in the complement
  of the one-dimensional theta series the exponents above would
  improve to
$\frac{k}{2}-\frac{1}{2}+\epsilon$ in the first case and to
$\frac{k}{2}-\frac{1}{4}+\epsilon$ in the second case above.
\item In the second case above the form $F$ is in the space of Maa\ss\
  or Saito-Kurokawa lifts, so that its Fourier coefficient at $T$  depends
  only on the determinant of $T$ and the summation over $T$ just
  multiplies the individual coefficient by the number of classes of $T$
  of determinant $d$.
\item Similar estimates can be obtained for vector valued Yoshida
  liftings associated to a pair of elliptic cusp forms $k_1,k_2$
  with $k_1>2,k_2>2$, using the results of
  \cite{bs_mellin_vector}. If such a vector valued lifting has
  transformation type ${\rm Sym}^j \otimes \det^k$ we obtain an
  exponent  $\frac{j+k}{2}-2\delta+\epsilon$ at $d$ with
  $\delta=\frac{1}{16}$ unconditionally and $\delta=\frac{1}{4}$ under
  the assumption of the generalized Ramanujan conjecture for cusp
  forms of half integral weight $\ge \frac{5}{2}$.

\parindent=0pt 
Similar estimates can also be obtained for Yoshida liftings of type
II, using Theorem \ref{average_asymptoticformula}.
   
\item A well known conjecture of B\"ocherer relates  $a(F,d)$ to the
  square of the central critical value of the twist by $\chi_{-d}$ of the
  spin $L$-function $Z_{\rm spin}(s,F)$ of $F$, if $F$ is a cuspidal
  Hecke eigenform for 
  the full Siegel modular group $Sp_2(\Z)$, more precisely it says
  that one should have
  \begin{equation*}
    Z_{\rm spin}(k-1, F, \chi_{-d})=c_F d^{1-k}(a(F,d))^2.
  \end{equation*}
It was shown in \cite{kriegraum} that the completed zeta function
$$\bigl(\frac{2\pi}{d}\bigr)^{-2s}\Gamma(s)\Gamma(s-k+2)Z_{\rm
  spin}(s,F, \chi_{-d})$$ has analytic continuation and satisfies a
functional equation under $s \mapsto 2k-2-s$. A subconvexity bound in
the conductor aspect for the central value of these twists would
therefore give an estimate 
of the type 
\begin{equation*}
   Z_{\rm spin}(k-1, F, \chi_{-d})=O(d^{1-\delta})
\end{equation*}
for some $\delta>0$.
 
Hence, if we assume B\"ocherer's conjecture and such a
  subconvexity bound for the central value of the twisted spin zeta
  function (with respect to the conductor 
  $d$ of the character) we get bounds for $a(F,d)$ of the above type
  (with $\delta$ in place of $1/8$);
  if we assume B\"ocherer's conjecture and the Lindel\"of hypothesis
  (again with respect to $d$) 
  for this central value we get the version given in a) which assumes the
  generalized Ramanujan conjecture for cusp forms of half integral
  weight. In the case that 
 the level of our modular forms is not $1$ we get the
  same type of estimate if we assume in addition to a generalized
  version of B\"ocherer's conjecture that the twisted spin zeta
  function has the same type of functional equation as given in
  \cite{kriegraum} in the case of level $1$ and satisfies the
  subconvexity bound mentioned above.  
  In the special case of the Yoshida liftings both
  the generalized version of B\"ocherer's conjecture and the required
  subconvexity estimate are known by \cite{bs_km}. 
\item    In \cite{kohnen} the Fourier coefficient at $T$ with $\det(T)=d$ of
   a cusp form of weight $k$ and degree $2$ for 
   the full Siegel modular group is estimated to be
   $O(d^{\frac{k}{2}-\frac{13}{36}+\epsilon})$. If we multiply this estimate by
     the number of integral equivalence classes of $T$ of determinant
     $d$ we obtain an exponent of $\frac{k}{2}+\frac{5}{36}+\epsilon$.
     The conjecture of Resnikoff and Salda\~{n}a \cite{res-sal} would
     give an exponent of 
     $\frac{k}{2}-\frac{3}{4}+\epsilon$ for the individual Fourier
     coefficient and hence an exponent of
     $\frac{k}{2}-\frac{1}{4}+\epsilon$ for the sum $a(F,d)$, which
     bound should be compared with the conjectural bound
     $\frac{k}{2}-\frac{1}{2}+\epsilon$ obtained above for Yoshida
     liftings outside the Maa\ss\ space under the assumption of the
     generalized Ramanujan conjecture for (good) cusp forms of half
     integral weight, or for general cusp forms orthogonal to the
     Maa\ss\  space under the assumption of B\"ocherer's conjecture and
     the Lindel\"of conjecture for the twisted spin zeta function. If
     we  consider only forms in  
     the complement of the Maa\ss\  space we see hence that our result for
     Yoshida liftings gives results for the averaged coefficient
     $a(F,d)$ that are significantly better than those obtained by
     multiplying the estimate for the individual Fourier coefficient
     by the number of summands. 
\end{enumerate}
 \end{myrem}
We might summarize the statements of the remark above in the
speculation that an estimate of the type 
\begin{equation}
  a(F,d)\le C d^{\frac{k}{2}-\delta +\epsilon}
\end{equation}
with some $\delta>0$ could be true and perhaps provable for a
general Siegel modular cusp form of degree $2$, including as in part
c) of the above remark the vector valued case. We might also speculate
that similar estimates may hold for averaged Fourier
coefficients of higher degree Siegel cusp forms.

Rainer Schulze-Pillot\\
Fachrichtung 6.1 Mathematik,
Universit\"at des Saarlandes (Geb. E2.4)\\
Postfach 151150, 66041 Saarbr\"ucken, Germany\\
email: schulzep@math.uni-sb.de
\end{document}